%% file: atom_spectrum_graded.tex
\documentclass[12pt]{amsart}

\usepackage[pdfauthor   = {Sebastian\ Posur},
            pdftitle    = {},
            pdfsubject  = {},
            pdfkeywords = {Atom spectrum;\ Cox ring;\ sheafification},
            bookmarks=true,
            bookmarksopen=true,
            pagebackref=true,
            hyperindex=true,
            colorlinks=true,
            linkcolor=blue,
            citecolor=blue,
            filecolor=blue,
            urlcolor=blue,
            ]{hyperref}

\include{header}

\author{Sebastian Posur}
\thanks{The author is supported by Deutsche Forschungsgemeinschaft (DFG) grant SFB-TRR 195: \emph{Symbolic Tools in Mathematics and their Application}}
\address{Department of mathematics, University of Siegen, 57068 Siegen, Germany}
\email{\href{mailto:Sebastian Posur <sebastian.posur@uni-siegen.de>}{sebastian.posur@uni-siegen.de}}

\begin{document}

\title[Atom spectra and sheafification in toric geometry]{Atom spectra of graded rings and sheafification in toric geometry}

\begin{abstract}
We prove that the atom spectrum, which is a topological space 
associated to an arbitrary abelian category introduced by Kanda,
of the category of finitely presented graded modules over a graded ring $R$
is given as a union of the homogeneous spectrum of $R$ with some additional points, which we call non-standard points.
This description of the atom spectrum helps in understanding
the sheafification process
in toric geometry: if $S$ is the Cox ring of a normal toric variety $X$ without torus factors,
then a finitely presented graded $S$-module sheafifies to zero
if and only if its atom support consists only of points
in the atom spectrum of $S$
which either lie in the vanishing locus
of the irrelevant ideal of $X$ or are non-standard.
\end{abstract}

\maketitle

\tableofcontents

\input{atom_spectrum_graded_content.tex}

\input{atom_spectrum_graded.bbl}

\end{document}

%% file: header.tex
\usepackage[utf8]{inputenc} 
\usepackage[T1]{fontenc}
\usepackage{a4wide}
\usepackage{lmodern}
\usepackage[english]{babel}
\usepackage{mathrsfs}
\usepackage{etex}
\usepackage{mathtools}
\usepackage{latexsym}
\usepackage{amssymb}
\usepackage{amsthm}
\usepackage{amsmath}
\usepackage{caption}
\usepackage{mathabx}

\usepackage{colortbl}

\usepackage[all]{xy}
\usepackage{verbatim}
\usepackage{listings}
\usepackage{fancyvrb}

\usepackage{graphicx}

\usepackage[dvipsnames]{xcolor}
\usepackage{accents} 
\usepackage{enumerate}
\usepackage{wrapfig}
\usepackage{tikz}
\usepackage{tikz-cd}
\usetikzlibrary{automata,shapes,arrows,matrix,backgrounds,positioning,plotmarks,calc,patterns,matrix,decorations.pathreplacing,decorations.pathmorphing,decorations.text,decorations.markings}
\usepackage[colorinlistoftodos,shadow]{todonotes}

\usepackage{multirow}
\usepackage{mdwlist}

\usepackage{stmaryrd}
\usepackage{mathdots} 

\usepackage{fancyvrb}

\usepackage[toc,page]{appendix}
\usepackage{float}

\usepackage{extarrows}
\usepackage{pdflscape}
\usepackage{rotating}
\usepackage{etex}
\newtheoremstyle{mytheoremstyle} 
    {5pt}                    
    {5pt}                    
    {\itshape}                   
    {\parindent}                           
    {\bf}                   
    {.}                          
    {.5em}                       
    {}  

\theoremstyle{mytheoremstyle}

\newtheorem{theorem}{Theorem}[section]
\newtheorem*{theoremnonumber}{Theorem}

\newtheorem{lemma}[theorem]{Lemma}

\newtheorem{theoremanddefinition}[theorem]{Theorem and Definition}
\newtheorem{definitionandlemma}[theorem]{Lemma and Definition}
\newtheorem{corollary}[theorem]{Corollary}

\newtheoremstyle{mytdefintionstyle} 
    {5pt}                    
    {5pt}                    
    {\rm}                   
    {\parindent}                           
    {\bf}                   
    {.}                          
    {.5em}                       
    {}  

\theoremstyle{remark}
\newtheorem{remark}[theorem]{Remark}
\newtheorem{caveat}[theorem]{Caveat}

\theoremstyle{mytdefintionstyle}
\newtheorem{definition}[theorem]{Definition}

\newtheorem{example}[theorem]{Example}

\newtheorem*{convention}{Convention}
\newtheorem{notation}[theorem]{Notation}

\newtheoremstyle{exmp_contd} 
{\topsep} {\topsep}%
{\upshape}
{}
{\bfseries}
{}
{ }
{\thmname{#1}\,\thmnumber{ #2}\thmnote{#3}\enspace(continued)}

\theoremstyle{exmp_contd}

\usepackage{xspace}
\input{pre.tex}

\tikzset{round left paren/.style={ncbar=0.5cm,out=120,in=-120}}
\tikzset{round right paren/.style={ncbar=0.5cm,out=60,in=-60}}
\newcolumntype{C}[1]{>{\centering\arraybackslash$}p{#1}<{$}}
\newlength{\mycolwd}
\usepackage{array, xcolor}
\definecolor{lightgray}{gray}{0.8}
\newcolumntype{L}{>{\raggedleft}p{0.28\textwidth}}
\newcolumntype{R}{p{0.8\textwidth}}

\definecolor{ctcolor}{gray}{0.95}

\definecolor{ctucolor}{gray}{0.85}

\makeatletter
\newcommand{\thickhline}{%
    \noalign {\ifnum 0=`}\fi \hrule height 1pt
    \futurelet \reserved@a \@xhline
}
\newcolumntype{"}{@{\hskip\tabcolsep\vrule width 1pt\hskip\tabcolsep}}
\makeatother

\usepackage{enumitem}
\newlist{theoremenumerate}{enumerate}{1}
\setlist[theoremenumerate]{label=(\arabic{theoremenumeratei}), ref=\thetheorem.(\arabic{theoremenumeratei}),noitemsep}

%% file: pre.tex
\definecolor{ExQ}{HTML}{0000FF}
\definecolor{Dec}{HTML}{E07B00}

\DeclareMathOperator{\Spec}{Spec}
\DeclareMathOperator{\ASpec}{ASpec}
\DeclareMathOperator{\Proj}{Proj}

\newcommand{\AC}{\mathbf{A}}

\newcommand{\KC}{\mathbf{K}}
\newcommand{\CC}{\mathbf{C}}

\newcommand{\N}{\mathbb{N}}
\newcommand{\Z}{\mathbb{Z}}

\newcommand{\Q}{\mathbb{Q}}

\newcommand{\V}{\mathbb{V}}
\newcommand{\C}{\mathbb{C}}
\newcommand{\pia}{\pi_{\mathrm{atom}}}

\newcommand{\Pro}{\mathbb{P}}

\newcommand{\pid}{\mathfrak{p}}

\newcommand{\maxid}{\mathfrak{m}}
\newcommand{\qid}{\mathfrak{q}}

\DeclareMathOperator{\Hom}{\mathrm{Hom}}

\DeclareMathOperator{\kernel}{\mathrm{ker}}

\DeclareMathOperator{\Supp}{\mathrm{Supp}}
\DeclareMathOperator{\ASupp}{\mathrm{ASupp}}

\DeclareMathOperator{\Sym}{\mathrm{Sym}}

\newcommand{\Pot}{\mathcal{P}}

\newcommand{\modl}{\text{-}\mathrm{mod}}

%% file: atom_spectrum_graded_content.tex
\section{Introduction}
The sheafification of a finitely presented graded module $M$ over the Cox ring $S$ of 
a normal toric variety $X$ over $\C$ without torus factors
yields a coherent sheaf $\widetilde{M}$ on $X$.
In the case of the $n$-dimensional projective space 
$\Pro^n = \Proj(\C[x_0, \dots, x_n])$, a finitely presented $\Z$-graded $S$-module $M$ sheafifies to zero if and only if
the support of $M$ lies in the vanishing locus of the irrelevant ideal $\langle x_0, \dots, x_n \rangle$.
However, this geometric criterion to check $\widetilde{M} \simeq 0$ does not naively generalize from $\Pro^n$ to $X$:
for example, it is impossible to read off from the support of
a graded module $M$ over the Cox ring of the weighted projective space $\Pro(1,1,2)$
whether $M$ sheafifies to zero or not \cite[Example 5.3.11]{CLS11}.
Cox states in \cite{Cox95} that
``there are many nonzero modules which give the zero sheaf''
and analyzes this ``phenomenon''
by giving an algebraic criterion for the 
special case where
$X$ is a simplicial toric variety
in \cite[Proposition 3.5]{Cox95}.

In this paper, we remedy the failure of the geometric criterion
for an arbitrary $X$
by applying the very general theory of atom spectra.
To an arbitrary abelian category $\AC$,
Kanda \cite{Kanda12} defines a topological space $\ASpec( \AC )$
called the atom spectrum of $\AC$
whose points consist of atoms, i.e., 
equivalence classes of so-called monoform objects.
For an object $A \in \AC$, its atom support $\ASupp(A) \subseteq \ASpec( \AC )$
consists of all the atoms defined by monoform subquotients of $A$.
A main theorem on atom spectra states that
if all objects in $\AC$ are noetherian,
then an \emph{open} subclass $U \subseteq \ASpec(\AC)$ 
defines a Serre subcategory
\[
 \{ A \in \AC \mid \ASupp(A) \subseteq U \},
\]
and all Serre subcategories arise in this way from a uniquely determined 
corresponding $U$.
Specializing to the case $\AC = S\modl$, the category of finitely presented graded $S$-modules,
it is natural to ask how the open subclass
corresponding to the Serre subcategory
\[
 \{ M \in S\modl \mid \widetilde{M} \simeq 0 \}
\]
looks like. 

To answer this question,
we first investigate the atom spectrum
of the category $R\modl$ of finitely presented $G$-graded $R$-modules for a $G$-graded ring $R$ in Section \ref{section:graded_atom},
where $G$ denotes an abelian group (see the beginning of that section
for our definition of a $G$-graded ring).
In Theorem \ref{theorem:main}, we describe the structure of $\ASpec(R) := \ASpec(R\modl)$
relative to the classical homogeneous spectrum of $R$:
any homogeneous prime ideal $\pid \subseteq R$ and any element $g \in G$
gives rise to a point $\widetilde{\pid}(g)$ in $\ASpec(R)$,
and two such points $\widetilde{\pid}(g)$ and $\widetilde{\qid}(h)$
are equal if and only if $\pid = \qid$ and
$h + G_{\pid} = g + G_{\pid}$, where $G_{\pid}$ is a subgroup
of $G$ determined by $\pid$ (see Definition \ref{definition:subgroup}).
We conclude in Remark \ref{remark:disjoint_union}
that we have a disjoint union on a set-theoretic level:
\[
 \ASpec(R) \simeq \big\{ \widetilde{\pid}(0) \mid \pid \subseteq R\text{~hom. prime} \big\} \uplus \big\{ \widetilde{\pid}(g) \mid g \not\in G_{\pid}, ~\pid \subseteq R\text{~hom. prime} \big\},
\]
where we call all points in that second subset \textbf{non-standard points}.

In Section \ref{section:applications}, we apply our findings
to the graded Cox ring $S$ of $X$ and remedy our criterion
for checking whether $\widetilde{M}\simeq 0$ in Theorem \ref{theorem:main3}:

\begin{theoremnonumber}
 Let $X$ be a normal toric variety over $\C$ without torus factors and with $G$-graded
 Cox ring $S$ and irrelevant ideal $B(\Sigma)$.
 A finitely presented $G$-graded $S$-module $M$ sheafifies to zero
 if and only if
 \[
  \ASupp( M ) \subseteq \big\{ \widetilde{\pid}(0) \mid \pid \in \Supp(S/B(\Sigma)) \big\} \uplus \big\{ \text{non-standard points in $\ASpec(S)$}\big\}.
 \]
\end{theoremnonumber}

In conclusion, we 
can think of the non-standard points
as the missing piece
of the homogeneous spectrum
for understanding when a graded module sheafifies
to zero by looking at its support.

 \begin{convention}
  Whenever we have an object equipped with a grading by an abelian group $G$, e.g., 
  a graded ring $R$ or a graded $R$-module $M$,
  we will simplify terminology as follows:
  \begin{enumerate}
   \item An element $r \in R$ will always mean a \emph{homogeneous} element.
         We write $\deg(r) \in G$ for its degree provided $r \neq 0$. We will do the same for elements $m \in M$.
   \item A prime ideal $\pid \subseteq R$ will always mean a \emph{homogeneous} prime ideal.
         We denote by $\Spec(R)$ the set of \emph{homogeneous} prime ideals of $R$.
   \item An $R$-module will always mean a \emph{graded} $R$-module,
         and likewise for submodules.
         By $R\modl$ we denote the category of finitely presented $G$-graded $R$-modules.
         We write $\Supp(M) \subseteq \Spec(R)$ for the \emph{homogeneous} primes in the support of $M$.
  \end{enumerate}
  Whenever we want to address the underlying non-graded ring of $R$ or non-graded module of $M$,
  we will write $|R|$ and $|M|$, and the adapted notions of element, prime ideal, spectrum, module,
  and support retain their classical meaning, e.g., by $\Spec(|R|)$, we address the
  set of not-necessarily homogeneous prime ideals of $R$.
  
  In this paper, we will make use of subsets of $G$ that constitute of the
  degrees of the homogeneous non-zero elements of a $G$-graded module $M$.
  We set
  \[
   \deg( M ) := \{ \deg(m) \mid m \in M\setminus \{0\} \text{~homogeneous} \} \subseteq G.
  \]
  
  If $M$ is a $G$-graded module,
  its \textbf{shift} by $g \in G$
  is the $G$-graded module $M(g)$ with
  graded parts
  \[
   M(g)_h := M_{g + h}
  \]
  for $h \in G$.
 \end{convention}

\section{Preliminaries: atom spectra and the classification of Serre subcategories}
We give a short introduction to Kanda's classification
of Serre subcategories of an abelian category consisting of noetherian objects \cite{Kanda12}.

\begin{definition}
 Let $\AC$ be an abelian category.
 A \emph{non-zero} object $A \in \AC$ is called \textbf{monoform}
 if the following holds:
 whenever we are given
 \begin{enumerate}
  \item a chain of subobjects $0 \lneq B' \leq B \leq A$,
  \item a subobject $C \leq A$,
  \item and an isomorphism $C \simeq B/B'$,
 \end{enumerate}
 then we already have $C \simeq 0$.
\end{definition}

Simple objects in $\AC$ are examples for monoform objects.
The free abelian group in one generator is an example for a non-simple monoform object
in the category of abelian groups.

\begin{remark}
 Non-zero subobjects of monoform objects are themselves monoform.
\end{remark}

We call two monoform objects $A,B \in \AC$
\textbf{atom-equivalent}
if they admit a common non-zero subobject.
Atom-equivalence defines an equivalence relation 
on the class of monoform objects in $\AC$ \cite[Proposition 2.8]{Kanda12}.
We denote the equivalence class of a monoform object $A \in \AC$ by $\overline{A}$
and call such a class an $\textbf{atom}$ of $\AC$.

\begin{definition}
 Let $\AC$ be an abelian category.
 \begin{enumerate}
  \item The \textbf{atom support} of an object $A \in \AC$
        is defined as
        \[
         \ASupp( A ) := \{ \overline{B} \mid B \text{~is a monoform subquotient of~} A\}
        \]
  \item The \textbf{atom spectrum} of $\AC$ is a topological space\footnote{Beware that its underlying class of points might not be a set.}
        whose points are given by
        \[
         \ASpec( \AC ) := \{ \overline{B} \mid B \text{~is a monoform object in~} \AC \}
        \]
        and with the classes $\ASupp(A)$, $A \in \AC$ as a basis of the topology.\footnote{ 
        This characterization of the opens in $\ASpec( \AC )$ is given in \cite[Proposition 3.9]{Kanda12}.
        }
 \end{enumerate}
\end{definition}

Recall that
an object $A \in \AC$ is called \textbf{noetherian}
if every ascending chain of its subobjects eventually stabilizes.
Now, we can turn to
the geometric picture
for the classification of Serre subcategories.

\begin{theorem}\label{theorem:Kandas_main_theorem}
 Let $\AC$ be an abelian category in which every object is noetherian.
 Then we get a bijection
 \begin{center}
         \begin{tikzpicture}[label/.style={postaction={
          decorate,
          decoration={markings, mark=at position .5 with \node #1;}},
          mylabel/.style={thick, draw=none, align=center, minimum width=0.5cm, minimum height=0.5cm,fill=white}}]
          \coordinate (r) at (2.5,0);
          \coordinate (d) at (0,-2);
          \node (A) {$\{ \text{Serre subcategories $\CC \subseteq \AC$}\}$};
          \node (B) at ($(A)+(d)$) {$\{ \text{open subclasses $U \subseteq \ASpec( \AC )$} \}$};
          
          \draw[bend right,->,thick,label={[left]{$\ASupp$}},out=360-25,in=360-155] (A) to (B);
          \draw[bend right,->,thick,label={[right]{$\ASupp^{-1}$}},out=360-25,in=360-155] (B) to (A);
          \draw[draw=none] (A) --node[rotate=90]{$\simeq$} (B);
          \end{tikzpicture}
        \end{center}
  where we set
  \[
   \ASupp( \CC ) := \bigcup_{C \in \CC} \ASupp(C)
  \]
  for a Serre subcategory $\CC \subseteq \AC$, and
  \[
   \ASupp^{-1}(U) := \{ A \in \AC \mid \ASupp( A ) \subseteq U \}
  \]
  for an open subclass $U \subseteq \ASpec( \AC )$.
  Under this bijection, the intersection of finitely many Serre subcategories
  corresponds to the intersection of open subclasses.
\end{theorem}
\begin{proof}
 The statement about the bijection is \cite[Theorem 4.3]{Kanda12} for classes instead of sets.
 Moreover, it is easy to see that $\ASupp^{-1}$ is compatible
 with intersections, and thus, the same is true for $\ASupp$.
\end{proof}

\begin{remark}\label{remark:geometric_criterion}
We get a geometric picture for 
the question whether a given object $A \in \AC$
lies in a Serre subcategory $\CC \subseteq \AC$:
\begin{align*}
 A \in \CC \hspace{1em}&\Longleftrightarrow\hspace{1em} A \in \ASupp^{-1}( \ASupp( \CC ) ) \\
 \hspace{1em}&\Longleftrightarrow\hspace{1em} \ASupp( A ) \subseteq \ASupp( \CC ).
\end{align*} 
\end{remark}

\begin{example}{\cite[Proposition 7.2, Remark 7.4]{Kanda12}}\label{example:commutative_noetherian}
 If $R$ is a commutative noetherian ring, 
 and $R\modl$ the category of finitely presented $R$-modules,
 then
 \[
  \Spec( R ) \stackrel{\sim}{\longrightarrow} \ASpec( R\modl ): \pid \mapsto \overline{R/\pid}
 \]
 defines a bijection that actually is a homeomorphism of topological spaces
 if we equip $\Spec( R )$ with the Hochster dual of the Zariski topology,
 i.e., the open sets in $\Spec( R )$ are exactly the specialization closed
 subsets.
 Furthermore, via this bijection,
 the atom support of a module $M \in R\modl$
 corresponds to its support. Thus, if $\CC \subseteq R\modl$
 denotes a Serre subcategory,
 then
 \[
  M \in \CC \hspace{1em}\Longleftrightarrow\hspace{1em} \Supp( M ) \subseteq \bigcup_{C \in \CC} \Supp( C ),
 \]
 a criterion which goes back to Gabriel \cite{Gab_thesis}.
\end{example}

\section{The atom spectrum of a graded ring}\label{section:graded_atom}

 Let $G$ be an (additively written) abelian group.
 By a \textbf{$G$-graded ring} $R$, we mean that
 \begin{enumerate}
  \item $R$ is a commutative ring,
  \item $R$ is equipped with a $G$-grading, i.e., a decomposition into abelian groups
 \[
  R = \bigoplus_{g \in G} R_g
 \]
 such that $R_g \cdot R_h \subseteq R_{g + h}$ for all $g,h \in G$,
 \item the subring $R_0$ is noetherian,
 \item there exist finitely many homogeneous elements $x_0, \dots, x_n \in R$
       for an $n \in \N_0$ such that $R = R_0[ x_0, \dots, x_n ]$.
 \end{enumerate}
Note that the underlying ring $|R|$ is noetherian since it is an algebra
of finite type over a noetherian ring.

We are going to describe 
\[
\ASpec( R ) := \ASpec( R\modl ). 
\]
We start by exhibiting a special class of monoform objects in $R\modl$.

\begin{lemma}\label{lemma:monoform_class}
 Let $\pid \in \Spec(R)$ and $g \in G$.
 Then $R/\pid(g)$ is a monoform object in $R\modl$.
\end{lemma}
\begin{proof}
  The functor that forgets the $G$-grading
  \[
   R\modl \rightarrow |R|\modl
  \]
  is faithful and exact, thus, it reflects monoform objects,
  i.e., $M \in R\modl$ is monoform if $|M|$ is.
  It follows that $R/\pid(g)$ is a monoform object
  since its underlying non-graded module (over the noetherian ring $|R|$) is monoform \cite[Proposition 7.1]{Kanda12}.
\end{proof}

Following the notation introduced in \cite[Section 6]{Kanda12}
and generalizing it to the graded case,
we write
\[
 \widetilde{\pid}(g) := \overline{R/\pid(g)} \in \ASpec(R)
\]
in order to refer to a point in the atom spectrum
associated to a prime $\pid \subseteq R$ and an element $g \in G$.
If $g$ is the neutral element, we simply write $\widetilde{\pid}$.

\begin{remark}\label{remark:opposed_to_non_graded}
 Opposed to the non-graded commutative case described in \cite[Proposition 7.1]{Kanda12}, 
 we can have ideals $\mathfrak{a} \subseteq R$ which
 are not prime, but such that $R/\mathfrak{a}$ is monoform.
 For example, consider the $\Z$-graded algebra $R = \Q[t]$ with $\deg(t) = 1$ and $\mathfrak{a} = \langle t^2 \rangle$.
 For $R/\mathfrak{a}$, the only non-trivial subquotient which is not also a subobject is given by $R/\langle t \rangle$.
 But there does not exist a submodule of $R/\mathfrak{a}$ which is isomorphic to $R/\langle t \rangle$.
 Nevertheless, since the submodule of $R/\mathfrak{a}$ which is generated by the class of $t$
 is isomorphic to $R/\langle t \rangle(-1)$,
 we have
 \[
  \overline{R/\mathfrak{a}} = \widetilde{\langle t \rangle}(-1)
 \]
 in $\ASpec(R)$.
\end{remark}

Next, we want to generalize the last observation of Remark \ref{remark:opposed_to_non_graded}
and see that the class of monoform
objects described in Lemma \ref{lemma:monoform_class} suffices to describe all atoms.
For this, we need the following well-known structure theorem
of finitely presented graded modules.

\begin{theorem}\label{theorem:structure_graded_module}
  Let $M \in R\modl$.
  Then there exist an $r \in \N_0$ and a filtration of $M$ by submodules
  \[
    0 = M^0 \leq M^1 \leq \dots \leq M^r = M
  \]
  such that $M^i/M^{i-1} \simeq R/\pid_i(g_i)$
  for $\pid_i \in \Spec(R)$ and $g_i \in G$, $i = 1, \dots, r$.
\end{theorem}
\begin{proof}
  The (non-constructive) proof in the $\Z$-graded case given in \cite[Proposition I.74]{Har}
  can be used by simply replacing $\Z$ with $G$.
\end{proof}
 
\begin{corollary}\label{corollary:monoform_reps}
  Let $M \in R\modl$ be monoform.
  Then there exist $\pid \in \Spec(R)$ and $g \in G$ such that
  \[
    \overline{M} = \widetilde{\pid}( g ).
  \]
\end{corollary}
\begin{proof}
  Theorem \ref{theorem:structure_graded_module}
  tells us that $M$ has a submodule of the form $R/\pid( g )$,
  which completes the proof.
\end{proof}

Thus, every atom of $R\modl$
is represented by a module of the form $R/\pid(g)$.
Next, we analyze when two such modules
give rise to the same atom.

\begin{definition}\label{definition:subgroup}
 Let $\pid \subseteq R$ be a prime.
 We set 
 \[
  G_{\pid} := \langle \deg(R/\pid) \rangle_{\Z} \leq G.
 \]
\end{definition}

\begin{lemma}\label{lemma:monoform_uniqueness}
 Given primes $\pid, \qid \subseteq R$ and elements $g,h \in G$.
 Then
 \[
  \widetilde{\pid}(g) = \widetilde{\qid}(h)
 \]
 if and only if 
 \[
 \pid = \qid \hspace{1.5em}\text{and}\hspace{1.5em} g + G_{\pid} = h + G_{\pid}.
 \]
\end{lemma}
\begin{proof}
 Since twisting with an element $g \in G$ defines an auto-equivalence of $R\modl$,
 it suffices to prove
 \[
  \widetilde{\pid}(g) = \widetilde{\qid} \hspace{1em}\Longleftrightarrow\hspace{1em} \pid = \qid \text{\hspace{0.5em}and\hspace{0.5em}} g \in G_{\pid}.
 \]
 Assume $\widetilde{\pid}(g) = \widetilde{\qid}$.
 Then $R/\pid( g )$ and $R/\qid$ have a common non-zero cyclic submodule
 whose annihilator is both $\pid$ and $\qid$. Thus, $\pid = \qid$.
 So, we are left to show
 \[
  \widetilde{\pid}(g) = \widetilde{\pid} \hspace{1em}\Longleftrightarrow\hspace{1em} g \in G_{\pid}.
 \]
 Let $x \in R \setminus \pid$.
 Then
 \[
  R/\pid(-\deg(x)) \simeq \langle \overline{x} \rangle_R \leq R/\pid,
 \]
 which implies
 \[
  \widetilde{\pid}( -\deg(x) ) = \widetilde{\pid}.
 \]
 By applying twists and using the transitivity of the atom-equivalence,
 we obtain
 \[
  \widetilde{\pid}(g) = \widetilde{\pid}
 \]
 for any $\Z$-linear combination $g$ of elements of the form $\deg(x)$ for $x \in R \setminus \pid$,
 in other words,
 for all $g \in G_{\pid}$.
 
 Conversely, if $\widetilde{\pid}(g) = \widetilde{\pid}$
 for some $g \in G$,
 then $R/\pid( g )$ and $R/\pid$ have a common non-zero cyclic subobject
 which is necessarily
 of the form $R/\pid(h)$ for some $h \in G$.
 Since $R/\pid(h) \hookrightarrow R/\pid$, we have $h \in G_{\pid}$.
 Since $R/\pid(h) \hookrightarrow R/\pid(g)$, we have $h-g \in G_{\pid}$
 and thus $g \in G_{\pid}$.
\end{proof}

\begin{lemma}\label{lemma:characterization_Gp}
 Let $\pid \subseteq R$ be a prime, and $I := \{ i \mid x_i \in \pid \}$.
 Then
 \[
  G_{\pid} = \langle \deg( x_i ) \mid i \not\in I \rangle_{\Z}.
 \]
 In particular, $G_{\pid}$ only depends on the set $I$.
\end{lemma}
\begin{proof}
 An element $r \in R/{\pid}$
 can be represented by an $R_0$-linear combination of monomials
 that only involve elements in $\{x_0,\dots,x_n\} \setminus \pid$.
 In particular, $\deg(r) \in \langle \deg(x) \mid x \in \{x_0,\dots,x_n\} \setminus \pid \rangle_{\Z}$.
\end{proof}

Now, we can fully understand the points in $\ASpec( R )$
and how they are related to the points in $\Spec(R)$.

\begin{theoremanddefinition}\label{theorem:main}
 Let $R$ be a $G$-graded ring. Then
 \[
  \ASpec(R) = \big\{ \widetilde{\pid}(g) \mid \pid \subseteq R \text{~prime}, ~g \in G \big\}.
 \]
 Moreover, the map
 \[
  \pia: \ASpec( R ) \rightarrow \Spec(R): \widetilde{\pid}(g) \mapsto \pid,
 \]
 which we call the \textbf{atom projection}, is well-defined and has the following properties:
 \begin{enumerate}
  \item For every $\pid \in \Spec(R)$, we have a bijection
  \[
   \pia^{-1}\big( \{\pid\} \big) \stackrel{\sim}{\longrightarrow} G/G_{\pid}: \widetilde{\pid}(g) \mapsto g + G_{\pid}.
  \]
  Moreover, the set of factor groups that occur as such fibers, i.e.,
 \[
  \big\{ G/G_{\pid} \mid \pid \subseteq R \text{~a prime} \big\},
 \]
  is finite.
  \item Given $M \in R\modl$, then
  \[
   \pia( \ASupp(M) ) = \Supp(M).
  \]
  \item Continuity in the following sense:
  given $M \in R\modl$, then 
  \[
  \pia^{-1}( \Supp(M) ) = \big\{ \widetilde{\pid}(g) \mid \pid \in \Supp(M), g \in G \big\}
  \]
   is an open subset of $\ASpec(R)$.
 \end{enumerate}
\end{theoremanddefinition}
\begin{proof}
 This theorem summarizes Lemma \ref{lemma:monoform_class}, Corollary \ref{corollary:monoform_reps}, and Lemma \ref{lemma:monoform_uniqueness}.
 Furthermore, from Lemma \ref{lemma:characterization_Gp}, 
 it follows that the number of different $G/G_{\pid}$ is bounded by
 the cardinality of the powerset of $\{x_0,\dots,x_n\}$.
 Thus, we only need to prove part $2$ and $3$ of the assertion.
 
 Part $2$ is easy for modules of the form $R/\pid(g)$ for primes $\pid \subseteq R$ 
 and elements $g \in G$.
 The general case
 follows from the structure Theorem \ref{theorem:structure_graded_module},
 and the compatibility of
 union of atom supports
 with extensions of objects \cite[Proposition 3.3]{Kanda12}.
 
 To prove continuity, 
 let $M \in R\modl$ and $\widetilde{\pid}(g) \in \pia^{-1}(\Supp(M))$. Then 
 \[
 \Supp(R/\pid(g)) \subseteq \Supp(M) 
 \]
  and thus
  \[
   \ASupp( R/\pid(g) ) \subseteq \pia^{-1}( \Supp( R/\pid(g) ) ) \subseteq \pia^{-1}(\Supp(M)).
  \]
  The claim follows since $\ASupp( R/\pid(g) )$ is an open neighborhood of $\widetilde{\pid}(g)$.\qedhere
\end{proof}

 The atom projection allows us to think about the atom spectrum of $R$
 as an enhancement of the homogeneous spectrum.
 To foster this intuition, we introduce the following notions.
 \begin{definition}\label{definition:standard_points}
  Let $R$ be a $G$-graded ring.
   We call an atom which is of the form $\widetilde{\pid}$ for a prime $\pid \subseteq R$
   a \textbf{standard point} in $\ASpec(R)$. All the other atoms, i.e., points of the form $\widetilde{\pid}(g)$ for $g \not\in G_{\pid}$,
   are called \textbf{non-standard points}.
   Furthermore, we set 
   \[
    \mathcal{U}_R := \big\{ \text{non-standard points in $\ASpec(R)$} \big\}.
   \]
 \end{definition}

 \begin{remark}\label{remark:disjoint_union}
  Theorem \ref{theorem:main} allows us to think of
  $\ASpec(R)$ as a disjoint union:
  \[
   \ASpec(R) \simeq \Spec(R) \uplus \mathcal{U}_R,
  \]
  where we identify $\Spec(R)$ with the set of all standard points.
 \end{remark}

\begin{lemma}
 Let $R$ be a $G$-graded ring. Then $\mathcal{U}_R \subseteq \ASpec(R)$ is open.
\end{lemma}
\begin{proof}
 Let $\widetilde{\pid}(g)$ be a non-standard point.
 It suffices to show that $\ASupp( R/\pid(g) ) \subseteq \mathcal{U}_R$,
 so, let $M$ be a monoform subquotient of $R/\pid(g)$.
 By Theorem \ref{theorem:structure_graded_module},
 $M$ contains a monoform submodule
 of the form $R/\qid(h)$ for
 $\qid \subseteq R$ a prime and $h \in G$, and $\overline{M} = \widetilde{\qid}(h)$.
 From
 \[
 \deg( R/\qid(h) ) \subseteq \deg( R/\pid(g) ),
 \]
 we conclude
 \[
  h \in g + G_{\pid} \neq G_{\pid}
 \]
 and thus 
 \[
 h \not\in G_{\pid} \supseteq G_{\qid}.
 \]
 It follows that $\widetilde{\qid}(h)$ is also a non-standard point.
\end{proof}

\begin{example}
Let $k$ be a field, and let
$R = k[ x_0, \dots, x_n ]$ be the 
$\Z$-graded polynomial ring with
all $x_i$ of degree $1$.
Set $\maxid := \langle x_0, \dots, x_n \rangle \in \Spec(R)$.
Then
\[
 \pia^{-1}(\{ \maxid \}) \simeq \Z
\]
and
\[
 \mathcal{U}_R = \big\{ \widetilde{\maxid}(i) \mid i \in \Z\setminus\{0\} \big\}.
\]
\end{example}

\begin{example}\label{example:aspec_p112}
Let $k$ be a field, and let
$R = k[ x_0, x_1, x_2 ]$ be the 
$\Z$-graded polynomial ring with
 $\deg(x_0) = \deg(x_1) = 1$ and $\deg(x_2) = 2$.
 Set
$\maxid := \langle x_0, x_1, x_2 \rangle$
and
$\pid := \langle x_0, x_1 \rangle$. Then
\[
 \pia^{-1}(\{ \maxid \}) \simeq \Z \hspace{1em}\text{and}\hspace{1em} \pia^{-1}(\{ \pid \}) \simeq \Z/2\Z
\]
and
\[
 \mathcal{U}_R = \big\{ \widetilde{\maxid}(i) \mid i \in \Z\setminus\{0\} \big\} \cup \big\{ \widetilde{\pid}(1) \big\}.
\]
\end{example}

\section{Applications to sheafification in toric geometry}\label{section:applications}

We are interested in the kernel of the sheafification process studied in toric geometry.
We give a short introduction to this subject
based on \cite[Chapter 5]{CLS11}, and then show how the atom
spectrum of a graded ring sheds new light on it.

For a normal toric variety $X$ over $\C$ without torus factors there exist
the following data:
\begin{enumerate}
 \item A natural number $n \in \N_0$ and an $n+1$-dimensional $\C$-vector space
       $W = \C^{n+1}$, whose standard basis elements are denoted by $x_0, \dots, x_n$.
 \item An algebraic subgroup $H \subseteq (\C^ {\ast})^{n+1}$ acting on $W$ by componentwise multiplication.
 \item A non-empty subset $\Sigma$ of the powerset of $\{0, \dots, n\}$ that is closed under taking subsets.
\end{enumerate}
An element $\sigma \in \Sigma$ gives rise to a monomial 
\[
 x^{\widehat{\sigma}} := \prod_{i \not\in \sigma} x_{i}
\]
in the so-called \textbf{Cox ring} of $X$, which is the symmetric algebra
\[
 S := \Sym( W ) = \C[ x_0, \dots, x_n ].
\]
The action of $H$ on $W$ induces an action on $S$,
and since $H$ is reductive, we get a decomposition
\[
 S = \bigoplus_{g \in G} S_g,
\]
where $G := \Hom_{\text{alg.Grps}}( H, \C^{\ast} )$ is the group of algebraic characters of $H$.
This turns $S$ into a $G$-graded ring with homogeneous generators $x_0, \dots, x_n$.

\begin{remark}
 The action of $H$ on $S$ gives the following geometric interpretation 
 of the group $G/G_{\pid}$ (see Definition \ref{definition:subgroup})
 for a prime $\pid \in \Spec(S)$.
 Set 
 \[
  I := \{ i \mid x_i \in \pid \} \hspace{2em}\text{and}\hspace{2em}\widehat{I} := \{ i \mid x_i \not\in \pid \}.
 \]
 We identify $(\C^{\ast})^I$ with the subgroup
 \[
  \big\{ (x_0, \dots, x_n) \mid \forall i \in \widehat{I}: x_i = 1 \big\} \subseteq (\C^{\ast})^{n+1}.
 \]
 Then, for any point
 \[
  p \in \big\{ (x_0, \dots, x_n) \in \C^{n+1} \mid x_i = 0 ~\Leftrightarrow~ i \in I\big\},
 \]
 we can compute the stabilizer subgroup of $H$ with respect to $p$ as
 \[
   \mathrm{Stab}_H(p) = \mathrm{Stab}_{(\C^{\ast})^{n+1}}(p) \cap H =\C^I \cap H.
 \]
 From the duality of the following
 pullback and pushout diagrams
 \begin{center}
  \begin{tabular}{ccc}

      \begin{tikzpicture}[label/.style={postaction={
        decorate,
        decoration={markings, mark=at position .5 with \node #1;}},
        mylabel/.style={thick, draw=none, align=center, minimum width=0.5cm, minimum height=0.5cm,fill=white}},
        baseline = ($0.5*(A) + 0.5*(C)$)]
        \coordinate (r) at (2.5,0);
        \coordinate (d) at (0,-2);
        \node (A) {$(\C^{\ast})^I$};
        \node (B) at ($(A) + (r)$) {$(\C^{\ast})^{n+1}$};
        \node (C) at ($(A) + (d)$) {$\mathrm{Stab}_H(p)$};
        \node (D) at ($(B) + (d)$) {$H$};
        
        \draw[right hook->,thick] (A) -- (B);
        \draw[right hook->,thick] (C) -- (A);
        \draw[right hook->,thick] (C) -- (D);
        \draw[right hook->,thick] (D) -- (B);
      \end{tikzpicture}
          
      & 
      \begin{tikzpicture}[label/.style={postaction={
        decorate,
        decoration={markings, mark=at position .5 with \node #1;}},
        mylabel/.style={thick, draw=none, align=center, minimum width=0.5cm, minimum height=0.5cm,fill=white}},
        baseline = (A)]
        \coordinate (r) at (2.5,0);
        \coordinate (d) at (0,-2);
        \node (A) {};
        \node (B) at ($(A) + (r)$) {};
        
        \draw[bend right,->,thick,label={[below]{$\Hom_{\text{alg.Grps}}( -, \C^{\ast} )$}},out=360-25,in=360-155] (A) to (B);
          \draw[bend right,->,thick,label={[above]{$\Hom_{\Z}( -, \C^{\ast} )$}},out=360-25,in=360-155] (B) to (A);
      \end{tikzpicture}

      & 
          
      \begin{tikzpicture}[label/.style={postaction={
        decorate,
        decoration={markings, mark=at position .5 with \node #1;}},
        mylabel/.style={thick, draw=none, align=center, minimum width=0.5cm, minimum height=0.5cm,fill=white}},
        baseline = ($0.5*(A) + 0.5*(C)$)]
        \coordinate (r) at (2.5,0);
        \coordinate (d) at (0,-2);
        \node (A) {$\Z^{n+1}/ \Z^{\widehat{I}}$};
        \node (B) at ($(A) + (r)$) {$\Z^{n+1}$};
        \node (C) at ($(A) + (d)$) {$G/G_{\pid}$};
        \node (D) at ($(B) + (d)$) {$G$,};
        
        \draw[->>,thick] (B) -- (A);
        \draw[->>,thick] (A) -- (C);
        \draw[->>,thick] (D) -- (C);
        \draw[->>,thick] (B) -- (D);
      \end{tikzpicture}
          
  \end{tabular}
 \end{center}
 we conclude that $G/G_{\pid}$ is the group of algebraic characters of $\mathrm{Stab}_H(p)$:
 \[
  G/G_{\pid} \simeq \Hom_{\text{alg.Grps}}\big( \mathrm{Stab}_H(p), \C^{\ast} \big).
 \] 
 As we have seen in Lemma \ref{lemma:characterization_Gp},
 $G/G_{\pid}$ only depends on $I$, not on $\pid$ itself.
\end{remark}

Next, we define the so-called \textbf{irrelevant ideal}
\[
 B( \Sigma ) := \langle x^{\widehat{\sigma}} \mid \sigma \in \Sigma \rangle_S
\]
of the Cox ring. 
Now, $X$ can be reconstructed as a variety as a quotient\footnote{In \cite[Chapter 5]{CLS11}, it is called an almost geometric quotient.}
\[
 X \simeq \left(\C^{n+1} \setminus \V\big( |B( \Sigma )| \big)\right)/\!\!/H,
\]
where $\V( |B( \Sigma )| )$ denotes the vanishing set of the irrelevant ideal (regarded without its grading),
which in particular implies that we have a bijection
\[
 \big\{ \text{ $\C$-points of $X$ } \big\} \simeq \big\{ \text{ closed orbits of the $H$-action on $\C^{n+1} \setminus \V\big( |B( \Sigma )| \big)$ } \big\}.
\]
\begin{example}\label{example:projective_space}
 For the projective space $\Pro^n$, we have $W = \C^{n+1}$,
 $H = \C^{\ast}$ acting diagonally on $W$,
 and $\Sigma = \Pot( \{ 0, \dots, n \} ) \setminus \{ \{ 0, \dots, n \} \}$,
 where $\Pot$ denotes the powerset.
 Then
 \[
  S = \C[ x_0, \dots, x_n ]
 \]
 is $G = \Hom( \C^{\ast}, \C^{\ast} ) \simeq \Z$ graded with $\deg( x_i ) = 1$ for all $i = 0, \dots n$,
 and
 \[
  B(\Sigma) = \langle x_0, \dots, x_n \rangle.
 \]
 We recover the usual description of $\Pro^n$ as a quotient:
 \[
  \Pro^n \simeq (\C^{n+1}-\{0\})/\!\!/\C^{\ast}.
 \]
\end{example}

Elements $\sigma \in \Sigma$
give rise to a cover of $X$ by affine open subschemes
\[
 \Spec( (S_{x^{\widehat{\sigma}}})_0 ) \subseteq X,
\]
where $(S_{x^{\widehat{\sigma}}})_0$ denotes the homogeneous localization\footnote{ 
The homogeneous localization of a graded ring $R$ at an element $x \in R$
consists of all degree $0$ elements in the localization $R_x$.
}
of $S$ at $x^{\widehat{\sigma}}$.
Now, a finitely presented $G$-graded module $M$ over $S$
induces a coherent sheaf $\widetilde{M}$ over $X$
which is uniquely determined by the property
\[
\widetilde{M}|_{\Spec( (S_{x^{\widehat{\sigma}}} )_0)} \simeq (M_{x^{\widehat{\sigma}}} )_0,
\]
where $(M_{x^{\widehat{\sigma}}} )_0$ denotes the homogeneous localization\footnote{ 
The homogeneous localization of a graded module $M$ over a graded ring $R$ at an element $x \in R$
consists of all degree $0$ elements in the localization $M_x$.
}
of $M$ at $x^{\widehat{\sigma}}$.
It is called the \textbf{sheafification of $M$}.

\begin{example}\label{example:projective_space_cont}
 We continue Example \ref{example:projective_space}.
 Given the $\Z$-graded $S$-module 
 \[
 k \simeq S/\langle x_0, \dots, x_n \rangle_S,
 \]
 we clearly have $\widetilde{k} \simeq 0$ since
 $k_{x_i} \simeq 0$ for all $i = 0, \dots, n$.
 More generally, we can say that a finitely presented $\Z$-graded $S$-module
 $M$ sheafifies to zero if and only if 
 \[
  \Supp( M ) \subseteq \{ \langle x_0, \dots, x_n \rangle \}.
 \]
 Note that it does not matter if we consider the support in $\Spec(S)$ or in $\Spec(|S|)$.
 Geometrically, this criterion seems very plausible since
 it is precisely the origin that is removed from the affine plane $\C^{n+1}$
 in the construction of $\Pro^n$.
 However, we will see in the next example that this criterion does not
 naively generalize to all toric varieties.
\end{example}

\begin{example}\label{example:p112_remedy}
 We analyze the weighted projective space $\Pro(1,1,2)$ \cite[Example 5.3.11]{CLS11}.
 In this case, $W = \C^3$, $H = \C^{\ast} \hookrightarrow (\C^{\ast})^3: h \mapsto (h,h,h^2)$,
 and $\Sigma = \Pot( \{ 0, 1, 2 \} ) \setminus \{ \{ 0, 1, 2 \} \}$.
 Then
 \[
  S = \C[ x_0, x_1, x_2]
 \]
 is $G = \Hom( \C^{\ast}, \C^{\ast} ) \simeq \Z$ graded with $\deg( x_0 ) = \deg( x_1 ) = 1$, $\deg( x_2 ) = 2$
 and
 \[
  B(\Sigma) = \maxid := \langle x_0, x_1, x_2 \rangle.
 \]
 We recover the usual description of $\Pro(1,1,2)$ as a quotient:
 \[
  \Pro(1,1,2) \simeq (\C^{3}-\{0\})/\!\!/\{(h,h,h^2)\mid h \in \C^{\ast}\}.
 \]
 Next, we set $\pid := \langle x_0, x_1 \rangle$ and take a look at the $S$-graded module
 \[
  M := S/\pid.
 \]
 We have $\widetilde{M} \not\simeq 0$
 since $(M_{x_2})_0 \simeq (k[x_2])_0 \simeq k.$
 However, if we shift $M$ by $1$, then we do have $\widetilde{M(1)} \simeq 0$:
 localizing $M$ at $x_0$ or $x_1$ yields the zero module, and
 homogeneous localization at $x_2$ yields:
 \[
  (M_{x_2}(1))_0 \simeq (k[x_2])_1 \simeq 0
 \]
 since $\deg( k[x_2] ) \subseteq 2\Z$.
 
 Since
 \[
  \Supp( M ) = \Supp( M(1) ) = \big\{ \pid, \maxid \big\}
 \]
 we cannot expect the support of a graded module 
 to give a criterion for the sheafification being zero,
 as it was the case in Example \ref{example:projective_space_cont}.
 
 If we take the atom supports instead, then we can understand the full picture.
 The preimages of $\maxid$ and $\pid$ under the atom projection are given by
 \[
  \pia^{-1}( \{\pid\} ) = \big\{ \widetilde{\pid}, \widetilde{\pid}(1) \big\}
 \]
 and
 \[
  \pia^{-1}(\{\maxid\} ) = \big\{ \widetilde{\maxid}(i) \mid i \in \Z \big\}
 \]
 (see Example \ref{example:aspec_p112}).
 Now, the atom supports of $M$ and $M(1)$ differ quite heavily:
 \[
 \ASupp( M ) = \big\{ \widetilde{\pid} \} \cup \{ \widetilde{\maxid}(i) \mid i \in 2\N_0 \big\} \neq 
 \big\{ \widetilde{\pid}(1) \} \cup \{ \widetilde{\maxid}(i) \mid i \in 2\N_0 - 1 \big\} = \ASupp( M(1) )
 \]
 If $\KC$ denotes the Serre subcategory of all those modules in $R\modl$ that sheafify to zero, then 
 the general criterion stated in Remark \ref{remark:geometric_criterion}
 tells us that an $S$-module $N$ sheafifies to zero if and only if
 \[
  \ASupp(N) \subseteq \ASupp( \KC ).
 \]
 Since we clearly have $\pia^{-1}(\{\maxid\}) \subseteq \ASupp( \KC )$,
 the only reason for $\widetilde{M} \not\simeq 0$, but $\widetilde{M(1)} \simeq 0$
 has to be
 \[
   \widetilde{\pid}(1) \in \ASupp( \KC ) \hspace{3em}\text{but}\hspace{3em} \widetilde{\pid} \not\in \ASupp( \KC ).
 \]
 The most apparent difference between these two points is that $\widetilde{\pid}$
 is a standard point while $\widetilde{\pid}(1)$ is non-standard.
\end{example}

The observations of Example \ref{example:p112_remedy}
motivate our main theorem of this section.
Recall that $\mathcal{U}_S$ denotes the set of non-standard points in $\ASpec(S)$.

\begin{theorem}\label{theorem:main3}
 Let $X$ be a normal toric variety over $\C$ without torus factors and with $G$-graded
 Cox ring $S$ and irrelevant ideal $B(\Sigma)$.
 A finitely presented $G$-graded $S$-module $M$ sheafifies to zero
 if and only if
 \[
  \ASupp( M ) \subseteq \big\{ \widetilde{\pid} \mid \pid \in \Supp(S/B(\Sigma)) \big\} \uplus \mathcal{U}_S.
 \]
\end{theorem}

We dedicate the subsections of this section
to the proof of Theorem \ref{theorem:main3}.

\begin{notation}
 Since the proof of Theorem \ref{theorem:main3}
 can be given in more general terms (see Theorem \ref{theorem:main2}),
 we let $G$ be an arbitrary abelian group,
 and $R$ be a $G$-graded ring in the sense of Section \ref{section:graded_atom}.
\end{notation}

For the proof of Theorem \ref{theorem:main3},
we proceed as follows. 
By definition, the kernel of the sheafification functor
is given as an intersection of kernels of various homogeneous localization
functors.
In Subsection \ref{subsection:general_description},
we give a general description of the atom support of the kernel of an exact functor
$F: R\modl \rightarrow \AC$ for an arbitrary abelian category $\AC$.
This description is made more concrete in the case
where $F$ is a homogeneous localization functor
in Subsection \ref{subsection:single_localization}.
Finally, in Subsection \ref{subsection:proof_main_theorem},
we investigate intersections of such atom supports
and prove an even more general version
of Theorem \ref{theorem:main3}.

\subsection{A general description of the atom support of the kernel of an exact functor}\label{subsection:general_description}

Given an exact functor
\[
 F: R\modl \rightarrow \AC
\]
into some abelian category $\AC$, its kernel is the Serre subcategory
\[
 \kernel( F ) := \{ M \in R\modl \mid F(M) \simeq 0 \}
\]
of $R\modl$.

The goal of this subsection is to prove the following general description
of the atom support of $\kernel(F)$.

\begin{theorem}\label{theorem:general_desc}
 Given an exact functor $F: R\modl \rightarrow \AC$, then
 \[
  \ASupp( \ker F ) = \big\{ \widetilde{\pid}(g) \in \ASpec(R) \mid F( R/\pid(g) ) \simeq 0 \big\}.
 \]
\end{theorem}

We will need the following lemma for its proof.

\begin{lemma}\label{lemma:aux}
 Let $\mathcal{U} \subseteq \Spec(R)$ be an open set.
 Then
 \[
  \mathcal{U} = \big\{ \widetilde{\pid}(g) \in \Spec(R) \mid \ASupp( R/\pid(g) ) \subseteq \mathcal{U} \big\}
 \]
\end{lemma}
\begin{proof}
We denote the set on the right hand side by $\mathcal{U}'$.

$\mathcal{U} \supseteq \mathcal{U}'$:
Given $\widetilde{\pid}(g) \in \Spec(R)$ such that $\ASupp( R/\pid(g) ) \subseteq \mathcal{U}$,
it follows that $\widetilde{\pid}(g) \in \mathcal{U}$ since $\widetilde{\pid}(g) \in \ASupp( R/\pid(g) )$.

$\mathcal{U} \subseteq \mathcal{U}'$:
Since $\mathcal{U}$ is open,
it suffices to prove that given
$M \in R\modl$ such that $\ASupp(M) \subseteq \mathcal{U}$,
we already have $\ASupp(M) \subseteq \mathcal{U}'$.
Since a module $M$ admits a finite filtration
by modules of the form $R/\pid(g)$ for primes $\pid \subseteq R$ 
and elements $g \in G$ (see Theorem \ref{theorem:structure_graded_module}),
and since the union of atom supports
is compatible with extensions of objects \cite[Proposition 3.3]{Kanda12},
we can assume $M = R/\pid(g)$.
For every subquotient $N$ of $M$, we have
\[
 \ASupp(N) \subseteq \ASupp(M) \subseteq \mathcal{U}.
\]
But this entails that every point in $\ASupp(M)$
already lies in $\mathcal{U}'$.
\end{proof}

\begin{caveat}
 The previous Lemma \ref{lemma:aux} does not imply that
 for every point $\widetilde{\pid}(g) \in \mathcal{U}$,
 we have $\ASupp( R/\pid(g) ) \subseteq \mathcal{U}$,
 but only that there exists $h \in g + G_{\pid}$
 such that $\ASupp( R/\pid(h) ) \subseteq \mathcal{U}$ (see Lemma \ref{lemma:monoform_uniqueness}).
\end{caveat}

\begin{proof}[Proof of Theorem \ref{theorem:general_desc}]
 Using Lemma \ref{lemma:aux}, we get
 \[
  \ASupp( \ker F ) = \big\{ \widetilde{\pid}(g) \in \ASpec(S) \mid \ASupp( R/\pid(g) ) \subseteq \ASupp( \ker F ) \big\}.
 \]
 But we know by Remark \ref{remark:geometric_criterion} that
 \[
  \ASupp( R/\pid(g) ) \subseteq \ASupp( \ker F )
  \hspace{1em}\Longleftrightarrow\hspace{1em}
  R/\pid(g) \in \ker F,
 \]
 which completes the proof.
\end{proof}

\subsection{The atom support of the kernel of the homogeneous localization functor}\label{subsection:single_localization}

Let $f \in R\setminus \{0\}$. We study the atom support of the kernel $\KC_f$ of
the exact homogeneous localization functor:
\[
 (-_f)_0: R\modl \rightarrow (R_f)_0\modl: M \mapsto (M_f)_0.
\]
Motivated by Theorem \ref{theorem:general_desc},
which tells us that
\[
 \ASupp( \KC_f ) = \big\{ \widetilde{\pid}(g) \mid (R/\pid(g)_f)_0 \simeq 0 \big\}
\]
we draw our attention to primes $\pid \subseteq R$
and elements $g \in G$ for which we have $R/\pid(g) \in \KC_f$.

\begin{lemma}\label{lemma:when_zero}
 Given a prime $\pid \subseteq R$ and an element $g \in G$, then
 \[
  R/\pid(g) \in \KC_f \hspace{1em}\Longleftrightarrow\hspace{1em} (f \in \pid) \vee \big(g \not\in \deg( R/\pid ) + \langle \deg(f) \rangle_{\Z}\big)
 \]
\end{lemma}
\begin{proof}
 If $f \in \pid$, then $(R/\pid)_f \simeq 0$, so in particular, its $g$-th degree part is zero.
 If $f \not\in \pid$, then for all $i \in \Z$,
 the element $f^i \in (R/\pid)_f$ is not zero, since $\pid$ is a prime.
 In particular,
 \[
  \deg\left(  R/\pid(g)_f \right) = \deg( R/\pid(g) ) + \langle \deg(f) \rangle_{\Z}.
 \]
 In this case, $(R/\pid(g)_f)_0 \simeq 0$ if and only if $0 \not\in \deg\left(  R/\pid(g)_f \right)$
 if and only if $0 \not\in \deg( R/\pid(g) ) + \langle \deg(f) \rangle_{\Z}$,
 which is equivalent to the condition in the claim.
\end{proof}

Due to Lemma \ref{lemma:when_zero}, we can write $\ASupp( \KC_f )$
as a union of two parts:
\begin{align*}
 \ASupp( \KC_f ) &= \big\{ \widetilde{\pid}(g) \mid (R/\pid(g)_f)_0 \simeq 0 \big\} \\
 &= \big\{ \widetilde{\pid}(g) \mid f \in \pid \big\} \cup \big\{ \widetilde{\pid}(g) \mid g \not\in \deg( R/\pid ) + \langle \deg(f) \rangle_{\Z} \big\} \\
 &= \pia^{-1}\big( \Supp( R/\langle f \rangle ) \big) \cup \big\{ \widetilde{\pid}(g) \mid g \not\in \deg( R/\pid ) + \langle \deg(f) \rangle_{\Z} \big\}
\end{align*}
The first subset is open by the continuity of the atom projection (Theorem \ref{theorem:main}).
The second subset in this union depends only on the degree of $f$, but not
on $f$ itself.

\begin{definitionandlemma}\label{lemma:open}
 For every $d \in G$, the subset
 \[
  \mathcal{U}_{d} := \big\{ \widetilde{\pid}(g) \mid g \not\in \deg( R/\pid ) + \langle \deg(f) \rangle_{\Z} \big\} \subseteq \ASpec(R)
 \]
 is open.
\end{definitionandlemma}
\begin{proof}
 Let $\pid \subseteq R$ be a prime and $g \in G$ such that $0 \not\in \deg( R/\pid(g) ) + \langle \deg(f) \rangle_{\Z}$.
 For any subquotient $M$ of $R/\pid(g)$, we have $\deg(M) \subseteq \deg( R/\pid(g))$ and thus
 \[
  0 \not\in \deg(M) + \langle d \rangle_{\Z} \subseteq \deg( R/\pid(g) ) + \langle d \rangle_{\Z}.
 \]
 It follows that the open neighborhood $\ASupp( R/\pid(g) )$ of $\widetilde{\pid}(g)$ lies in $\mathcal{U}_{d}$.
\end{proof}

In summary, we get the following description of $\ASupp( \KC_f )$:

\begin{theorem}\label{theorem:kernel_decomposition}
 The atom support of the kernel $\KC_f$ of the homogeneous localization functor $(-_f)_0$
 can be written as the following union of open subsets:
 \[
  \ASupp( \KC_f ) = \pia^{-1}\big( \Supp( R/\langle f \rangle ) \big) \cup \mathcal{U}_{\deg(f)}.
 \]
\end{theorem}

\begin{remark}
 The open set
 \[
  \pia^{-1}\big( \Supp( R/\langle f \rangle ) \big)
 \]
 in Theorem \ref{theorem:kernel_decomposition}
 corresponds to the kernel of $(-_f)$, i.e.,
 localization without taking the $0$-th degree part.
 So, the occurrence of $\mathcal{U}_{\deg(f)}$
 is a peculiarity of \emph{homogeneous} localization.
\end{remark}

\subsection{Proof of the main theorem}\label{subsection:proof_main_theorem}

In this subsection, we prove Theorem \ref{theorem:main2},
which is a variant of the main Theorem \ref{theorem:main3},
but with the advantage that it also works outside the context of
toric geometry, and that it easily implies the main Theorem \ref{theorem:main3}.

Remember that we have $R = R_0[x_0, \dots, x_n]$ a $G$-graded
ring in the sense of Section \ref{section:graded_atom}.
For a subset  $\sigma \subseteq \{0, \dots, n\}$, we define $\widehat{\sigma} := \{0, \dots, n\} \setminus \sigma$ and
\[
x^{\widehat{\sigma}} := \prod_{i \not\in \sigma} x_i \in R. 
\]

\begin{lemma}\label{lemma:technical_lemma}
 Let $\sigma \subseteq \{0, \dots, n\}$ and assume that $x^{\widehat{\tau}} \neq 0$ for all $\tau \subseteq \sigma$.
 Then
 \[
  \bigcap_{\tau \subseteq \sigma} \ASupp\big( \KC_{x^{\widehat{\tau}}} \big) = \pia^{-1}\big( \Supp( R/\langle x^{\widehat{\sigma}} \rangle ) \big) \cup \mathcal{U}_R,
 \]
 where we set $\KC_{x^{\widehat{\tau}}} := \kernel( (-_{x^{\widehat{\tau}}})_0 )$.
\end{lemma}

\begin{remark}
 Note that the statement of Lemma \ref{lemma:technical_lemma}
 is a peculiarity of \emph{homogeneous} localization.
 The intersection of
 the kernels of the functors $(-_{x^{\widehat{\tau}}})$ for all $\tau \subseteq \sigma$, i.e., without taking the $0$-th degree part,
 would simply be equal to $\kernel( (-_{x^{\widehat{\sigma}}}) )$.
\end{remark}

\begin{proof}[Proof of Lemma \ref{lemma:technical_lemma}]
 Throughout this proof, we will think of $\ASupp( \KC_{x^{\widehat{\tau}}} )$
 as the union
 \[
  \ASupp( \KC_{x^{\widehat{\tau}}} ) = \pia^{-1}\big( \Supp( R/\langle x^{\widehat{\tau}} \rangle ) \big) \cup \mathcal{U}_{\deg(x^{\widehat{\tau}})}
 \]
 that was given by Theorem \ref{theorem:kernel_decomposition}.
 Now, call $\mathcal{I}_l$ the left hand side of the equation in the statement of the lemma and $\mathcal{I}_r$ the right
 hand side.
 
 $\mathcal{I}_l \supseteq \mathcal{I}_r$:
 Since 
 \[
 \Supp( R/\langle x^{\widehat{\sigma}} \rangle ) \subseteq \Supp( R/\langle x^{\widehat{\tau}} \rangle ) 
 \]
 for all $\tau \subseteq \sigma$, 
 it follows that 
 \[
  \pia^{-1}\big( \Supp( R/\langle x^{\widehat{\sigma}} \rangle ) \big) \subseteq \mathcal{I}_l.
 \]
 Next, let $\widetilde{\pid}(g)$ be a non-standard point.
 Given $\tau \subseteq \sigma$, we distinguish two cases.
 First, assume that there exists $j \in \widehat{\tau}$ such that $x_j \in \pid$.
 In this case, 
 \[
  \widetilde{\pid}(g) \in \pia^{-1}\big( \Supp( R/\langle x^{\widehat{\tau}} \rangle ) \big) \subseteq \ASupp( \KC_{^{\widehat{\tau}}} ).
 \]
 Second, assume that $x_j \not\in \pid$ for all $j \in \widehat{\tau}$.
 In that case,
 \[
  g \not\in G_{\pid} = \langle \deg( R/ \pid) \rangle_{\Z} \supseteq \deg( R/ \pid) + \langle \deg( x^{\widehat{\tau}} ) \rangle_{\Z}
 \]
 since $\deg( x^{\widehat{\tau}} ) = \sum_{j \in {\widehat{\tau}}}\deg(x_j) \in \langle \deg( R/ \pid) \rangle_{\Z}$ (note that $x^{\widehat{\tau}} \neq 0$).
 Thus, 
 \[
 \widetilde{\pid}(g) \in \mathcal{U}_{\deg(x^{\widehat{\tau}})} \subseteq \ASupp( \KC_{{\widehat{\tau}}} ).
 \]
 
 $\mathcal{I}_l \subseteq \mathcal{I}_r$:
 The idea is to prove that given an element of the form
 \[
  \widetilde{\pid}(g) \in \ASupp( \KC_{x^{\widehat{\sigma}}} ) \setminus \mathcal{I}_r,
 \]
 we can find a $\tau \subseteq \sigma$ such that $\widetilde{\pid}(g) \not\in \ASupp( \KC_{x^{\widehat{\tau}}} )$.
 
 Since $\widetilde{\pid}(g) \in \mathcal{U}_{\deg( x^{\widehat{\sigma}} )}$, we may assume that $g$ is chosen such that
 \[
  g \not\in \deg( R/\pid ) + \langle \deg( x^{\widehat{\sigma}} ) \rangle_{\Z}.
 \]
 Since $\widetilde{\pid}(g)$ is a standard point, $g \in G_{\pid}$, and thus
 \[
 g \in G_{\pid} \setminus ( \deg( R/\pid ) + \langle \deg( x^{\widehat{\sigma}} ) \rangle_{\Z} ).
 \]
 Next, we set $\tau := \{ i \in \{0, \dots n\} \mid x_i \in \pid \}$. We have $\tau \subseteq \sigma$
 since $\pid \not\in \Supp( R/\langle x^{\widehat{\sigma}} \rangle )$ and thus $x_i \not\in \pid$ for all $i \in \widehat{\sigma}$.
 Now, it suffices to prove that 
 \[
  \widetilde{\pid}(g) \not\in \ASupp( \KC_{x^{\widehat{\tau}}} ).
 \]
 Clearly, $\widetilde{\pid}(g) \not\in \pia^{-1}\big( \Supp( R/\langle x^{\widehat{\tau}} \rangle ) \big)$ since $x_j \not\in \pid$ for all $j \in \widehat{\tau}$.
 So, let us assume that we have $\widetilde{\pid}(g) \in \mathcal{U}_{\deg(x^{\widehat{\tau}})}$. Then
 \[
  \exists h \in g + G_{\pid} = G_{\pid}
 \]
 such that
 \[
  h \not\in \deg( R/\pid ) + \langle \deg(x^{\widehat{\tau}}) \rangle_{\Z}.
 \]
 But we can compute
 \begin{align*}
  \deg( R/\pid ) + \langle \deg(x^{\widehat{\tau}}) \rangle_{\Z} &= \langle \deg( x_j ) \mid j \in \widehat{\tau} \rangle_{\N_0} + \langle \sum_{i \in \widehat{\tau}}\deg( x^i ) \rangle_{\Z}\\
  &= \langle \deg( x_j ) \mid j \in \widehat{\tau} \rangle_{\Z}\\
  &= G_{\pid}
 \end{align*}
 and thus get $h \not\in G_{\pid}$, a contradiction.
\end{proof}

We are ready to prove our generalized version of Theorem \ref{theorem:main3}.

\begin{theorem}\label{theorem:main2}
 Let $R$ be $G$-graded ring (in the sense of Section \ref{section:graded_atom}).
 Let $\Sigma$ be a subset of the powerset of $\{0, \dots, n\}$ that is closed under taking subsets.
 Assume that
 \[
  x^{\widehat{\sigma}} := \prod_{i \not\in \sigma}x_i \neq 0
 \]
 for all $\sigma \in \Sigma$.
 Then 
 \[
  \bigcap_{\sigma \in \Sigma} \ASupp\big( \KC_{x^{\widehat{\sigma}}} \big) = \pia^{-1}\big( \Supp( R/B(\Sigma) ) \big) \cup \mathcal{U}_R,
 \]
 where we set $\KC_{x^{\widehat{\sigma}}} := \kernel( (-_{x^{\widehat{\sigma}}})_0 )$
 and
 $B(\Sigma) := \langle x^{\widehat{\sigma}} \mid \sigma \in \Sigma \rangle$.
\end{theorem}
\begin{proof}
 \begin{align*}
  \bigcap_{\sigma \in \Sigma} \ASupp\big( \KC_{x^{\widehat{\sigma}}} \big) &=
  \bigcap_{\sigma \in \Sigma} \bigcap_{\tau \subseteq \sigma} \ASupp\big( \KC_{x^{\widehat{\tau}}} \big) &\\
  &= \bigcap_{\sigma \in \Sigma} \big( \pia^{-1}\big( \Supp( R/\langle x^{\widehat{\sigma}} \rangle ) \big) \cup \mathcal{U}_R \big)& \text{Lemma \ref{lemma:technical_lemma}} \\
  &= \big( \bigcap_{\sigma \in \Sigma}  \pia^{-1}\big( \Supp( R/\langle x^{\widehat{\sigma}} \rangle )\big) \big) \cup \mathcal{U}_R & \\
  &= \big( \pia^{-1}\big( \bigcap_{\sigma \in \Sigma} \Supp( R/\langle x^{\widehat{\sigma}} \rangle )\big) \big) \cup \mathcal{U}_R & \\
  &= \pia^{-1}\big( \Supp( R/B(\Sigma) ) \big) \cup \mathcal{U}_R & \qedhere\\
 \end{align*}
\end{proof}

\begin{proof}[Proof of Theorem \ref{theorem:main3}]
 The kernel $\KC$ of the sheafification functor is given by
 \[
  \KC = \bigcap_{\sigma \in \Sigma} \kernel( (-_{x^{\widehat{\sigma}}})_0 ).
 \]
 Thus,
 \[
  \ASupp( \KC ) = \bigcap_{\sigma \in \Sigma} \ASupp( \kernel( (-_{x^{\widehat{\sigma}}})_0 ) ),
 \]
 and from Theorem \ref{theorem:main2}, we conclude
 \[
  \ASupp( \KC ) = \pia^{-1}\big( \Supp( S/B(\Sigma) ) \big) \cup \mathcal{U}_S.
 \]
 Since
 \[
  \pia^{-1}\big( \Supp( S/B(\Sigma) ) \big) \cup \mathcal{U}_S =
  \big\{ \widetilde{\pid} \mid \pid \in \Supp(S/B(\Sigma)) \big\} \uplus \mathcal{U}_S,
 \] 
 we can finish the proof with
 the geometric criterion stated in Remark \ref{remark:geometric_criterion}.
\end{proof}

\nocite{FAC}

%% file: atom_spectrum_graded.bbl
\def\cprime{$'$} \def\cprime{$'$} \def\cprime{$'$} \def\cprime{$'$}
  \def\cprime{$'$}
\providecommand{\bysame}{\leavevmode\hbox to3em{\hrulefill}\thinspace}
\providecommand{\MR}{\relax\ifhmode\unskip\space\fi MR }
\providecommand{\MRhref}[2]{%
  \href{http://www.ams.org/mathscinet-getitem?mr=#1}{#2}
}
\providecommand{\href}[2]{#2}